\documentclass{amsart}
\usepackage{amsthm}
\usepackage{tikz}
\usepackage{mathrsfs}
\usetikzlibrary{calc,decorations.markings}

\newtheorem{thm}{Theorem}
\newtheorem{lem}{Lemma}

\newtheorem{defi}{Definition}

\newtheorem{asump}{Assumption}
\newtheorem{appli}{Application}

\def\Re{\text{Re}}

\def\d{\mathrm{d}}

\begin{document}
\author{Pranendu Darbar}
\address{Institute of Mathematical Sciences, 
CIT Campus, Taramani, Chennai 600113, India}
\email[Pranendu Darbar]{dpranendu@imsc.res.in}

\begin{abstract}
In this paper, we find asymptotic formula for  the following sum with explicit error term:
\[M_{x}(g_{1}, g_{2}, g_3)=\frac{1}{x}\sum_{n\le x}g_{1}(F_1(n))g_{2}(F_2(n))g_{3} (F_3(n)),\]
where $F_1(x), F_2(x)$ and $F_3(x)$ are polynomials with integer coefficients and $g_1,g_2,g_3$ are multilpicative functions with modulus less than or equal to $1.$

Moreover, under some assumption on $g_1,g_2,$ we prove that as $x\rightarrow \infty,$
\[\frac{1}{x}\sum\limits_{n\le x}g_1(n+3)g_2(n+2)\mu(n+1)=o(1)\]
and assuming $2$-point Chowla type conjecture we show that as $x\rightarrow \infty,$
\[\frac{1}{x}\sum\limits_{n\le x}g_1(n+3)\mu(n+2)\mu(n+1)=o(1).\]
 
\end{abstract}

\title{Triple Correlations of Multiplicative Functions}
\maketitle

\vskip 4mm

\begin{center} \section{Introduction} \end{center}

  Let $ g_{j} : \mathbb{N} \rightarrow \mathbb{C} $ be multiplicative functions such that $|g_j(n)|\le 1$ for all $n$ and $j=1,2,3.$ Let $F_1(x),F_2(x),F_3(x)$ are polynomials with integer coefficients.
 
Consider the following triple correlation function:
\begin{eqnarray}
M_{x}(g_{1}, g_{2}, g_{3})=\frac{1}{x}\sum_{n\le x}g_{1}(F_1(n))g_{2}(F_2(n))g_{3}(F_3(n)),
\end{eqnarray}

In \cite{KAT}, Kat\'{a}i studied the  asymtotic bahaviour of the above sum $(1)$ when $F_j(x),j=1,2,3$ are special polynomials and some assumptions on $g_j,j=1,2,3,$ but did not provide error term. In \cite{ST4}, Stepanauskas studied the asymptotic formula for sum $(1)$ with explicit error term when $F_j(x),j=1,2,3$ are linear polynomials and  $g_1,g_2,g_3$ are close to $1$ (see definition $1$). Recently, Klurman \cite{KL} studied the double correlation function (i.e. the sum $(1)$ with $g_3=1$).

Estimations of $(1)$ can be used to get information on the behaviour of the limit distribution of the sum
\begin{align}
f_1(F_1(n))+f_2(F_2(n))+f_3(F_3(n)),
\end{align}
where $f_1,f_2$ and $f_3$ are real-valued additive functions. 

From onwards, let $F(n);F_1(n),F_2(n),F_3(n)$ be positive integer-valued polynomials with integer coefficients and these are not divisible by the square of any irreducible polynomial. Also suppose that $F_j(n),F_k(n)$ are relatively primes for $j\neq k $ and for all $n$.
Let $v$ and $v_j$ denote the degree of the polynomials $F(n)$ and $F_j(n)$ respectively.

In this paper, we will investigate the following sums with various assumptions on $g_j,j=1,2,3$:
\begin{eqnarray}
M_x^{'}(g_1,g_2,g_3)=\frac{1}{x}\sum\limits_{n\le x}g_1(n+3)g_2(n+2)g_3(n+1)
\end{eqnarray} 
and asymptotic formula of the following triple correlation function with explicit error term which is a improvement of a theorem of Kat\'{a}i (\cite{KAT}, Theorem $5$) with respect to explicit error term:
\begin{eqnarray}
M_x^{''}(g_1,g_2,g_3)=\frac{1}{x}\sum\limits_{n\le x}g_1(F_1(n))g_2(F_2(n))g_3(F_3(n)),
\end{eqnarray}
where $F_1(x),F_2(x)$ and $F_3(x)$ are polynomials as above of degree $\ge 2.$
\begin{defi}
A multiplicative function $g_j$ is said to be close to $1$ if 
\begin{align}
\sum\limits_{p}\frac{g_j(p)-1}{p}<\infty.
\end{align}
\end{defi}
\begin{defi}
A multiplicative function $g_j$ is called good function if there exists a $\kappa\in \mathbb{C}$ such that for each $u>0$
\begin{align}
\sum\limits_{p\le x}|g_j(p)-\kappa|\ll \frac{x}{(\log x)^u}.
\end{align}
\end{defi}
\subsection*{Outline}
In section $5,$ we formulate the sums $(3)$ and $(4)$ in terms of main term and explicit error term.

For the asymptotic behaviour of the sum $(3)$ it seems to be very difficult if we take  either $g_1, g_2,$ and $g_3$ are equal to Mobius function or any two of them are equal to Mobius function and other one is close to some fixed complex number. 

In section $6,$ firstly we investigate the asymptotic behaviour of the sum $(3)$ when $g_1, g_2$ are close to $1$ and $g_3$ is the Mobius function. Secondly, we investigate the sum $(3)$ in terms of main term and explicit error term when $g_3$ is a good function.

In section 7, under some assumption we investigate the asymtotic behaviour of the sum $(3)$ when $g_1$ is close to $1$ and $g_2,g_3$ are Mobius function.

In section 8, we formulate some applications the above sums $(3)$ and $(4).$
\subsection*{Acknowledgement}
 I would like to express my  appreciation to my thesis supervisor Anirban Mukhopadhyay for his valuable and constructive suggestions during the planning and development of this paper. I thank Gediminas Stepanauskas for providing me with  preprint of \cite{ST3},\cite{ST4},\cite{JS}.

\section{\textbf{Notations}}
Throughout the paper $p$ and $q$ denote primes; $j,k,l,m$ and $n$ are natural numbers;  $c,c^{'},c_1,c_2,\cdots$ are absolute constants; $c_F,c_{F_1},c_{F_2},c_{F_3}$ are constants depending on $F, F_1, F_2, F_3$ and $\wp$ is the set of all primes. 

In section 3, we  use the following notations for Theorem $1$ :
\begin{align*}
&S(r,x,g_{j})=\sum_{r<p\le x+4-j}\frac{|g_{j}(p)-1|^{2}}{p},\quad S(r,x)=\sum_{j=1}^{3}S(r,x,g_{j}),r>0,\\
&P(x)=\prod\limits_{p\le x}w_{p} \quad \text{and}\quad P(r,x)=\prod_{r<p\le x} w_{p},
\end{align*}
where
\begin{eqnarray}
w_{p}=\underset{\substack{\left(p^{m_1},p^{m_2}\right)=1\\ \left(p^{m_2},p^{m_3}\right)=1\\ \left(p^{m_1},p^{m_3}\right) |2}}{\sum\limits_{\substack{m_{1}=0 \\ \left(p^{m_1},3\right)=1}} ^{\infty}\sum\limits_{\substack{m_{2}=0\\ \left(p^{m_2},2\right)=1}}^{\infty}\sum\limits_{m_{3}=0}^{\infty}
}\frac{h_{1}(p^{m_1})h_{2}(p^{m_2})h_{3}(p^{m_3})}{[p^{m_1},p^{m_2},p^{m_3}]}.
 \end{eqnarray}
We  use the following notations for Theorem $2$ :

Let $\varrho(d)$ and $\varrho_j(d)$ denote the number of congruent solutions of the congruences $F(n)\equiv 0\pmod d$ and $F_j(n)\equiv 0\pmod d$ respectively.

Let $\varrho(d_1,d_2,d_3)$ be the number of solutions of the congruence system
\[F_j(n)\equiv 0\pmod {d_j}\quad j=1,2,3.\]
Let $D_\gamma $ denote the set of those tuples $\left\lbrace d_1,d_2,d_3\right\rbrace$ of natural numbers, for which all the prime factors of $d_i$ do not exceed $\gamma.$
\begin{align*}
&S^{'}(r,x)=\sum\limits_{r<p\le x}\sum\limits_{j=1}^{3}\frac{\left|g_j(p)-1\right|^2\varrho_j(p)}{p},\quad T(x)=\sum\limits_{j=1}^{3}\sum\limits_{x<p\le F_j(x)}\frac{\left|g_j(p)-1\right|^2\varrho_j(p)}{p},\\
&C(r)=\sum\limits_{j=1}^{3}\sum\limits_{m=1}^{v_j-1}\sum\limits_{\substack{p^m\le F_j(x)\\p>r}}\left|g_j(p^m)-1\right|\varrho_j(p^m),\\
&P^{'}(x)=P_1(\gamma)P_2(\gamma, x) \quad \text{and} \quad P^{'}(r,x)=\prod_{r<p\le x} w_{p}^{'},
\end{align*}
where
\begin{align}
&P_1(\gamma)=\sum\limits_{\left\lbrace d_1,d_2,d_3\right\rbrace \in D_\gamma} \frac{h_1(d_1)h_2(d_2)h_3(d_3)}{[d_1,d_2,d_3]}\varrho(d_1,d_2,d_3),\\
&P_2(\gamma,x)=\prod\limits_{\gamma <p\le x}\biggl(1+\sum\limits_{m=1}^{\infty}\sum\limits_{j=1}^{3}\frac{h_j(p^m)\varrho_j(p^m)}{p^m}\biggr):=\prod\limits_{\gamma <p\le x}w_{p}^{'}.
\end{align}
For Theorem $3$ and $4$ we need the following notations:
 \begin{align*}
 &\theta_{\tau}(n)=\prod\limits_{p|n}\biggl(1+\sum\limits_{m=1}^{\infty}\frac{g_3(p^m)}{p^{m(1+i\tau)}}\biggr)^{-1},\tau \in \mathbb{R},\\
 &Q(r)=\prod\limits_{p\le r}\biggl(1-\frac{2\theta_{\tau}(p)}{p-1}+\theta_{\tau}(p)\sum\limits_{m=1}^{\infty}\frac{g_1(p^m)+g_2(p^m)}{p^m}\biggr),\\
 &P^{''}(r,x)=\prod\limits_{r<p\le x}\left(1-\frac{1}{p}\right)\biggl(1+\sum\limits_{m=1}^{\infty}\frac{g_1(p^m)+g_2(p^m)}{p^m}\biggr),\\
 &S^{''}(r,x)=\sum\limits_{j=1}^{2}\sum\limits_{r<p\le x+4-j}\frac{|g_j(p)-1|^2}{p}\quad \text{and} \quad M_x(g_3)=\frac{1}{x}\sum\limits_{n\le x}g_3(n).
 \end{align*}  
 \section{\textbf{Statements of Theorems}}
We begin with the asymtotic formula for the sum $(3)$ with explicit error term which is a special case of a theorem of Stepanauskas(\cite{ST4}).
\begin{thm}
Let $g_1$, $g_2$ and $g_3$ be multiplicative functions with modulus less than or equal to $1$. Then there exists a positive absolute constant $c$ such that for all $x\ge r\ge 2$ and for all $\frac{2}{3}<\alpha <1,$ we have
\[M_{x}^{'}(g_{1}, g_{2}, g_{3})=P(x)+  O \Bigl(x^{2-3\alpha}\exp\Bigl(c\frac{r^{\alpha}}{\log r}\Bigr)+(S(r,x))^{\frac{1}{2}}+(r\log r)^{-\frac{1}{2}}\Bigr).\]
\end{thm}
The aim of this paper is to prove the following statements:
\begin{thm}
Let $F_j(x),j=1,2,3$ be polynomials as above of degree greater than or qual to $2.$ Let $g_1, g_2$ and $g_3$ be multiplicative functions whose modulus do not exceed $1.$ Then there exists a positive absolute constant $c$ and a natural number $\gamma$ such that for all $x\ge r\ge \gamma$ and for all $1-\frac{1}{v_1+v_2+v_3}<\alpha <1,$ we have
\begin{eqnarray*}
M_x^{''}(g_1,g_2,g_3)-P^{'}(x)&\ll & \left(F_1(x)F_2(x)F_3(x)\right)^{1-\alpha}\exp\left(c\frac{r^\alpha}{\log r}\right)+(S^{'}(r,x))^{\frac{1}{2}}\\
&+&(T(x))^{\frac{1}{2}}+(r\log r)^{-\frac{1}{2}}+\frac{1}{x}C(r)+\frac{1}{\log x}.
\end{eqnarray*}
\end{thm}
\begin{thm}
Let $g_1$ and $g_2$ be multiplicative functions which do not exceed $1$ and 
\begin{eqnarray}
\sum\limits_{p}\sum\limits_{i=1}^{2}\frac{|g_{i}(p)-1|^2}{p}< \infty.
\end{eqnarray}
Then as $x\rightarrow \infty$,
\[M_{x}^{'}(g_{1}, g_{2}, \mu)=\frac{1}{x}\sum_{n\le x}g_{1}(n+3)g_{2}(n+2)\mu (n+1)=o(1).\] 
\end{thm}
\begin{thm}
Let $g_1, g_2$ and $g_3$ be multiplicative functions whose modulus do not exceed $1$ and $g_3$ is a good function.

Assume further that there exist a positive constant $c_1$ such that 
\begin{align}
\biggl|1+\sum\limits_{k=1}^{\infty}\frac{g_3(2^k)}{2^{k(1+i\xi)}}\biggr|\ge c_1
\end{align}
 for $\xi=0,$ if $g_3$ is real valued and for all $\xi \in \mathbb{R},$ if $g_3$ is not real valued. 
Then there exists a positive absolute constant $c$ and a real $\tau,|\tau|\le (\log x)^{1/19},$ such that for all $x\ge r\ge 2$ and for all $\frac{1}{2}<\alpha<\frac{5}{9},$ we have
\begin{align*}
&M_x^{'}(g_1,g_2,g_3)-M_x(g_3)P^{''}(r,x)Q(r)\ll x^{1-2\alpha}\exp\left(c\frac{r^{\alpha}}{\log r}\right)+\frac{(\log r)^{c}}{(\log x)^{c^{'}}}\\
&+\frac{\exp\left(c(\log\log r)^2\right)}{(\log x)^{1/19}}+(S^{''}(r,x))^{\frac{1}{2}}+(r\log r)^{-\frac{1}{2}}.
\end{align*}
For real-valued $g_3$ we may set $\tau =0.$
\end{thm}
\begin{asump}[$2$-point Chowla type conjecture]
For every given $A>0$,
\[\sum\limits_{n\le x}\mu (n+2)\mu (n+1)\exp(2\pi in\alpha)=O\left(\frac{x}{(\log x)^A}\right)\]
holds uniformly for all real $\alpha.$ 
\end{asump}
\begin{thm}
Let $g_1$ be a multiplicative function such that $|g_1(n)|\le 1$ for all $n$ and  
\begin{eqnarray}
\sum\limits_{p}\frac{|g_1(p)-1|^2}{p}< \infty.
\end{eqnarray}
Suppose that Assumption $1$ holds, then as $x\rightarrow \infty$,
\[M_{x}^{'}(g_1, \mu, \mu)=\frac{1}{x}\sum_{n\le x}g_1(n+3)\mu(n+2)\mu (n+1)=o(1).\]
\end{thm}
\section{\textbf{Applications}}
In this section we state several applications of our results.\\
Application $1$ is a linear version of a theorem of Kat\'{a}i (\cite{KAT},Theorem $1$).
\begin{appli} Let $g_1, g_2, g_3$ be multiplicative function such that for all $j=1,2,3,$ $|g_j|\le 1$ and $g_1,g_2,g_3$ are close to $1.$
Then as $x\rightarrow \infty$ we have,
\[M_{x}^{'}\left(g_1, g_2,g_3\right) = \prod\limits_{p}w_p+o(1),\]
where $w_p$ is defined in $(7).$
\end{appli}
Application $2$ is a polynomial version with degree of polynomial greater than or equal to $2$ of a theorem of Kat\'{a}i(\cite{KAT},Theorem $5$).
\begin{appli}
Let $F_j(n)(j=1,2,3)$ be as above of degree $v_j\ge 2$.Let $g_j(j=1,2,3)$ be as above and
\begin{align}
\sum\limits_{p}\frac{(g_j(p)-1)\varrho_j(p)}{p}<\infty.
\end{align}
Suppose that as $p\rightarrow \infty,$
\begin{align}
\left(g_j(p^\alpha)-1\right)\varrho_j(p^\alpha)\rightarrow 0,
\end{align} 
for $\alpha=1,$ when $v_j\ge 2$ and for $\alpha=1,2,\cdots,v_j-2,$ when $v_j\ge 3,$ then we have as $x\rightarrow\infty,$
\[M_x^{''}(g_1,g_2,g_3)\rightarrow P_1(\gamma)P_2(\gamma),\]
where $P_1(\gamma)$ is defined by $(8)$ and 
\begin{align}
P_2(\gamma):=\prod\limits_{p>\gamma}\biggl( 1+\sum\limits
_{m=1}^{\infty}\sum\limits_{i=1}^{3}\frac{h_i(p^m)\varrho_i(p^m)}{p^m}\biggr).
\end{align}
\end{appli}
Applications $3,4,5$ and $6$ are the direct apllications of the Theorems $1,3,5$ and $2$ respectively.
\begin{appli} 
Let $\phi(n)=n\prod\limits_{p|n}\left(1-\frac{1}{p}\right)$, be Euler's totient function and $\sigma_{a} (n)=\sum\limits_{d|n}d^a, a>0$. Then for $x\ge 2,0<A<1$ we have,
\begin{align*}
&\frac{1}{x}\sum\limits_{n\le x}\frac{\sigma_{a} (n+3)\sigma_{a} (n+2)\sigma_{a} (n+1)}{(n+3)^{a}(n+2)^{a}(n+1)^{a}}=w_2\prod\limits_{p>2}\left\lbrace 1+\frac{3}{p^{a+1}-1}\right\rbrace + O\left(\frac{1}{(\log x)^A}\right),\\
&\frac{1}{x}\sum\limits_{n\le x}\frac{\phi(n+3)\phi(n+2)\phi(n+1)}{(n+3)(n+2)(n+1)}=w_2\prod\limits_{p>2}\left\lbrace1-\frac{3}{p^2}\right\rbrace + O\left(\frac{1}{(\log x)^A}\right),
\end{align*}
where $w_2$ is defined by $(7)$ at $p=2$ in which $g_j,j=1,2,3$ are replaced by $\sigma_a$ and $\phi$ respectively to the above sums.
\end{appli}
\begin{appli}
If $\phi, \mu$ and $\sigma_{a},a>0$ are as above then as $x\rightarrow \infty$,
\begin{align*}
&\frac{1}{x}\sum_{n\le x}\frac{\phi(n+3)}{(n+3)}\frac{\phi(n+2)}{(n+2)}\mu (n+1)=o(1),\\ 
&\frac{1}{x}\sum_{n\le x}\frac{\sigma_{a}(n+3)}{(n+3)^a}\frac{\sigma_{a}(n+2)}{(n+2)^a}\mu (n+1)=o(1).
\end{align*} 
\end{appli}
\begin{appli}
If $\phi, \mu$ and $\sigma_{a},a>0$ are as above and under Assumption $1$, we have as $x\rightarrow\infty,$
\begin{align*}
&\frac{1}{x}\sum_{n\le x}\frac{\phi(n+3)}{(n+3)}\mu(n+2)\mu (n+1)=o(1),\\
&\frac{1}{x}\sum_{n\le x}\frac{\sigma_{a}(n+3)}{(n+3)^a}\mu(n+2)\mu (n+1)=o(1).
\end{align*}
 \end{appli}
 \begin{appli}
 Let $F_1(x)=x^2+b,F_2(x)=x^2+c,F_3(x)=x^2+d, a>0,0<t<1,$
 where $b,c,d$ are taken such that $F_j(x),j=1,2,3$ satisfies the assumption of Theorem $2$ and quadratic residue for all odd prime $p.$ Then there exist a natural number $\gamma$ such that for all $x\ge \gamma,$
 \begin{align*}
 &\frac{1}{x}\sum\limits_{n\le x}\frac{\sigma_a(n^2+b)\sigma_a(n^2+c)\sigma_a(n^2+d)}{\left(n^2+b\right)^a\left(n^2+c\right)^a\left(n^2+d\right)^a}=P_1^{'}(\gamma)\prod\limits_{p
 >\gamma}\left(1+\frac{6}{p^{a+1}-1}\right)+O\left(\frac{1}{(\log x)^t}\right)\\
 &\frac{1}{x}\sum\limits_{n\le x}\frac{\phi(n^2+b)\phi(n^2+c)\phi(n^2+d)}{(n^2+b)(n^2+c)(n^2+d)}=P_1^{''}(\gamma)\prod\limits_{p>\gamma}\left(1-\frac{6}{p^2}\right)+O\left(\frac{1}{(\log x)^t}\right),
 \end{align*}
 where $P_1^{'}(\gamma)$ and $P_2^{''}(\gamma)$ are defined by $(8)$ in which $g_j,j=1,2,3$ are replaced by $\sigma_a$ and $\phi$ respectively.
 \end{appli}
 Applications $7$ and $8$ are the behaviour of the distribution of the sum $(2).$
 \begin{appli}
 Let $f_1, f_2$ and $f_3$ be real-valued additive functions and 
\begin{align}
 &\sum\limits_{|f_j(p)|\le 1}\frac{f_{j}^2 (p)}{p}<\infty, j=1,2,3,\\
 &\sum\limits_{|f_j(p)|> 1}\frac{1}{p}<\infty, j=1,2,3,\\
 &\sum\limits_{j=1}^{3}\sum\limits_{|f_j(p)|\le 1 }\frac{f_{j}(p)}{p}<\infty,
\end{align}
Then the distribution functions
 \begin{eqnarray}
 \frac{1}{[x]}\# \left\lbrace n|n\le x, f_1(n+3)+f_2(n+2)+f_3(n+1)\le z\right\rbrace
 \end{eqnarray}
converge weakly towards a limit distribution $($\cite{TEN}, Chapter III.$2)$ as $x\rightarrow \infty,$ and the characteristic function of this limit distribution is equal to  
 \begin{eqnarray}
 w_2\prod\limits_{p>2}\biggl(1-\frac{3}{p}+\left(1-\frac{1}{p}\right)\sum\limits_{m=1}^{\infty}\sum\limits_{k=1}^{3}
 \frac{\exp\left(itf_k(p^m)\right)}{p^m}\biggr),
 \end{eqnarray}
 where $w_2$ is defined by $(7)$ at $p=2$ with $g_k$ is replaced by $\exp(itf_k),k=1,2,3.$
 \end{appli}
 \begin{appli}
 Let $f_1, f_2$ and $f_3$ be real-valued additive functions and $F_j(n),j=1,2,3$ are as above of degree $v_j\ge 2$. Assume that
\begin{align}
 &\sum\limits_{|f_j(p)|\le 1}\frac{f_{j}^2 (p)}{p}\varrho_j(p)<\infty, j=1,2,3,\\
 &\sum\limits_{|f_j(p)|> 1}\frac{\varrho_j(p)}{p}<\infty, j=1,2,3,\\
 &\sum\limits_{j=1}^{3}\sum\limits_{|f_j(p)|\le 1}\frac{f_{j}(p)\varrho_j(p)}{p}<\infty,\\
 &f_j(p^m)\varrho_j(p^m)\rightarrow 0,
 \end{align} 
 for $m=1,$ when $v_j\ge 2$ and for $m=1,2,\cdots,v_j-2,$ when $v_j\ge 3.$
Then the distribution functions
 \begin{eqnarray}
 \frac{1}{[x]}\# \left\lbrace n|n\le x, f_1(F_1(n))+f_2(F_2(n))+f_3(F_3(n))\le z\right\rbrace
 \end{eqnarray}
converge weakly towards a limit distribution as $x\rightarrow \infty,$ and the characteristic function of this limit distribution is equal to  $P_1(\gamma)P_2(\gamma),$
 where $P_1(\gamma)$ and $P_2(\gamma)$ are defined by $(8)$ and $(15)$ respectively with $g_j$ is replaced by $\exp(itf_j),j=1,2,3.$
 \end{appli}
 
\section{\textbf{Proof of Theorem $\bf 2$}} 
We begin with some  lemmas. The first lemma is required for the polynomial version of classical Tur\'{a}n-Kubilius inequality.
\begin{lem}[\cite{ERD}, Lemma 3]
Let, $F(m)$ be arbitrary primitive polynomial of degree $v$ with integer coefficients and with discriminant $D.$ Let, $D\neq 0.$ Then the number of solution of the congruence $F(m)\equiv 0 \pmod{p^{\alpha}}$ is $\varrho(p),$ when $p\not | D,$ and smaller than $vD^2$ when $p|D.$

Further, $\varrho(ab)=\varrho(a)\varrho(b),$ if $(a, b)=1.$\\
and $\varrho(p^{\alpha})\le c,$ $c$ depends only on $F.$
\end{lem}
Now we prove a polynomial version of classical Tur\'{a}n-Kubilius inequality.
\begin{lem}
 Let, $f(p^m)$ be the sequence of complex numbers for all $p\in \wp$, $m\ge 1$ and $F(n)$ is a polynomial as above of degree $v.$ Then we have
  \[\sum\limits_{n\le x}\left|K(F(n))-A(x)\right|\ll xB(F(x))+\sum\limits_{m=1}^{v-1}\sum\limits_{p^m\le F(x)}\left|f(p^m)\right|\varrho(p^m)+\frac{x}{\log x},\]
where
\begin{align*}
& K(n):=\sum\limits_{p^m\parallel n}f(p^m), \quad A(x):=\sum\limits_{p^m\le x}\frac{f(p^m)\varrho (p^m)}{p^m},\\
& B^2(x):=\sum\limits_{p^m\le x}\frac{|f(p^m)|^2\varrho(p^m)}{p^m}.
\end{align*}
\begin{proof}
We write
$K(F(n))=\sum\limits_{p^m\parallel F(n)}f(p^m)=g(F(n))+h(F(n)),$\\
where
\begin{align*}
g(y)=\sum\limits_{\substack{p^m\parallel y\\p^m\le y}}f(p^m)\quad \text{and} \quad h(y)=\sum\limits_{\substack{p^m\parallel y\\p^m>y}}f(p^m).
\end{align*}
Now,
\[\sum\limits_{n\le x}|K(F(n))-A(x)|\le \sum\limits_{n\le x}\left|g(F(n))-A(x^{1/2})\right|+\sum\limits_{n\le x}\left|h(F(n))\right|+\sum\limits_{n\le x}\left|A(x^{1/2})-A(x)\right|.\]
From Tur\'{a}n-Kubilius inequality(\cite{ELL},Lemma $4.11$), we have
\[\sum\limits_{n\le x}\bigl|g(F(n))-A(x^{1/2})\bigr|\ll xB(x^{1/2}).\]
From Lemma $1$ and Cauchy-Schwarz inequality, we have
\begin{align*}
&\left|A(x)-A(x^{1/2})\right|\le \sum\limits_{x^{1/2}<p^m\le x}\frac{\left|f(p^m)\right|\varrho(p^m)}{p^m}\\
&\le \biggl(\sum\limits_{x^{1/2}<p^m\le x}\frac{\left|f(p^m)\right|^2\varrho(p^m)}{p^m}\biggr)^{1/2}\biggl(\sum\limits_{x^{1/2}<p^m\le x}\frac{\varrho(p^m)}{p^m}\biggr)^{1/2}=O(B(x)).
\end{align*}
Again by Cauchy-Schwarz inequality, we have
\begin{align*}
&\sum\limits_{n\le x}|h(F(n))|=\sum\limits_{n\le x}\biggl|\sum\limits_{\substack{p^m\parallel F(n)\\p^m>x^{1/2}}}f(p^m)\biggr|\\
& \ll \sum\limits_{x^{1/2}<p^m\le F(x)}\frac{x|f(p^m)|\varrho(p^m)}{p^m}+\sum\limits_{x^{1/2}<p^m\le F(x)}\left|f(p^m)\right|\varrho(p^m)\\
 &\ll x\biggl(\sum\limits_{x^{1/2}<p^m\le F(x)}\frac{|f(p^m)|^2\varrho(p^m)}{p^m}\biggr)^{1/2}\biggl(\sum\limits_{x^{1/2}<p^m\le F(x)}\frac{\varrho(p^m)}{p^m}\biggr)^{1/2}\\
 &+\sum\limits_{x^{1/2}<p^m\le F(x)}\left|f(p^m)\right|\varrho(p^m)
\ll xB(F(x))+\sum\limits_{m=1}^{v-1}\sum\limits_{p^m\le F(x)}\left|f(p^m)\right|\varrho(p^m)+\frac{x}{\log x}.
\end{align*}
Which proves the required Lemma.
\end{proof}
\end{lem}

The following lemma ensures the existance of $\gamma$ in Theorem $2.$
\begin{lem}[\cite{TAN}, Lemma 2.1]
If $F_1(m)$ and $F_2(m)$ are relatively prime polynomials with integer coefficients, then the congruence $F_1(m)\equiv 0 \pmod{a}, F_2(m)\equiv 0 \pmod{a}$ have common roots atmost for finitely many values of $a.$ 
\end{lem}
Define multiplicative functions $g_{jr}$ and $g_{jr}^{*}$, $j=1,2,3$ by\\
\begin{align*}
g_{jr}(p^m)=
\begin{cases}
g_j(p^m) \ & \text{if } p\le r \\
1 \ & \text{if } p>r.
\end{cases} \quad \text{and} \quad
g_{jr}^{*}(p^m)=
\begin{cases}
1 \ & \text{if } p\le r \\
g_j(p^m) \ & \text{if } p>r.
\end{cases}
\end{align*}
and multiplicative function $h_{jr}$, $j=1,2,3$ by 
\begin{equation*}
h_{jr}(p^m)=
\begin{cases}
g_j(p^m)-g_j(p^{m-1}) \ & \text{if } p\le r \\
0 \ & \text{if } p>r.
\end{cases}
\end{equation*}
so that, $g_{jr}=1\ast h_{jr}$, $j=1,2,3.$
\subsection{Proof of Theorem $2$}
We can write
\begin{align*}
M_{x}^{''}(g_{1}, g_{2}, g_3)-P^{'}(x)=P^{'}(r,x)\biggl(\frac{1}{x}\sum_{n\le x}g_{1r}(F_1(n))g_{2r}(F_2(n))g_{3r}(F_3(n))-P^{'}(r)\biggr)\\ 
+\frac{1}{x}\sum_{n\le x}g_{1r}(F_1(n))g_{2r}(F_2(n))g_{3r}(F_3(n))\left(g_{1r}^{\ast}(F_1(n))g_{2r}^{\ast}(F_3(n))g_{3r}^{\ast}(F_3(n))-P^{'}(r,x)\right).
\end{align*}
So,
\begin{align*}
&\left|M_x^{''}(g_1,g_2,g_3)-P^{'}(x)\right|\le \frac{1}{x}\biggl|\sum\limits_{n\le x}g_{1r}(F_1(n))g_{2r}(F_2(n))g_{3r}(F_3(n))-P^{'}(r)\biggr|\\
&+\frac{1}{x}\sum\limits_{n\le x}\left|g_{1r}^{\ast}(F_1(n))g_{2r}^{\ast}(F_2(n))g_{3r}^{\ast}(F_3(n))-P^{'}(r,x)\right|:=S_1+S_2. 
\end{align*}
\subsection*{Estimation of $S_1$}
\begin{align*}
&\frac{1}{x}\sum\limits_{n\le x}g_{1r}(F_1(n))g_{2r}(F_2(n))g_{3r}(F_3(n))=\frac{1}{x}\sum\limits_{\substack{d_j\le F_j(x)\\j=1,2,3}}h_{1r}(d_1)h_{2r}(d_2)h_{3r}(d_3)\sum\limits_{\substack{n\le x\\d_j|F_j(x)\\j=1,2,3}}1\\
&= \frac{1}{x}\sum\limits_{d_1\le F_1(x)}\sum\limits_{d_2\le F_2(x)}\sum\limits_{d_3\le F_3(x)}h_{1r}(d_1)h_{2r}(d_2)h_{3r}(d_3)\frac{x}{[d_1,d_2,d_3]}\varrho(d_1,d_2,d_3)\\
&+O\Biggl(\frac{1}{x}\sum\limits_{\substack{d_i\le F_j(x)\\j=1,2,3}}h_{1r}(d_1)h_{2r}(d_2)h_{3r}(d_3)\varrho(d_1,d_2,d_3)\Biggr) :=P_1^{'}+S_3.
\end{align*}
Now we observe that
\[ \sum\limits_{d_j=1}^{\infty}\frac{|h_{jr}(d_j)|\varrho_j(d_j)}{d_j}\le \exp\biggl(c_{F_j} \sum\limits_{p\le r}\frac{1}{p}\biggr)\ll (\log r)^{c_{F_j}}\]
and for $0<\alpha <1$ 
\[\sum\limits_{d_j =1}^{\infty}\frac{|h_{jr}(d_j)|\varrho_j(d_j)}{d_j^{1-\alpha}}\le \prod\limits_{p\le r}\left(1+\sum\limits_{m=1}^{\infty}\frac{|h_{jr}(p^m)|\varrho_j(p^m)}{p^{m(1-\alpha)}}\right)\le \exp\left(c_{F_j}\frac{r^\alpha}{\log r}\right).\]
We can say that,
\begin{eqnarray*}
S_3&\ll & \frac{1}{x}\sum\limits_{\substack{d_j\le F_j(x)\\j=1,2,3}}|h_{1r}(d_1)h_{2r}(d_2)h_{3r}(d_3)|\varrho(d_1)\varrho(d_2)\varrho(d_3)\\
&\ll &\frac{1}{x}\left(F_1(x)F_2(x)F_3(x)\right)^{1-\alpha}\sum\limits_{\substack{d_j=1\\j=1,2,3
}}^{\infty}\frac{|h_{1r}(d_1)|}{d_1^{1-\alpha}}\frac{|h_{2r}(d_2)|}{d_2^{1-\alpha}}\frac{|h_{3r}(d_3)|}{d_3^{1-\alpha}}\varrho(d_1)\varrho(d_2)\varrho(d_3)\\
&\ll &\frac{1}{x}\left(F_1(x)F_2(x)F_3(x)\right)^{1-\alpha}\exp(c_F\frac{r^\alpha}{\log r}).
\end{eqnarray*}
Now,
\begin{eqnarray*}
P_1^{'}&=&P^{'}(r)+O\Biggl(\sum\limits_{k=1}^{3}\sum\limits_{\substack{d_j=1\\j=1,2,3\\
d_k>F_k(x)}}^{\infty}\frac{|h_{1r}(d_1)h_{2r}(d_2)h_{3r}(d_3)|}{[d_1,d_2,d_3]}\varrho(d_1,d_2,d_3)\Biggr)\\
&:=&P^{'}(r)+S_4.
\end{eqnarray*}
Again from the above observations, we have
\begin{eqnarray*}
 S_4&\ll &\sum\limits_{k=1}^{3}\sum\limits_{\substack{d_j=1\\j=1,2,3\\d_k>F_k(x)}}^
 {\infty}\frac{|h_{1r}(d_1)h_{2r}(d_2)h_{3r}(d_3)|}{d_1d_2d_3}\varrho(d_1)\varrho(d_2)\varrho(d_3)\\
&\ll&\left(F_1(x)^{-\alpha}+F_2(x)^{-\alpha}+F_3(x)^{-\alpha}\right)
\exp\left(c_F\frac{r^\alpha}{\log r}\right).
\end{eqnarray*}
\subsection*{Estimation of $S_2$}
Let
\begin{eqnarray*}
N_r^{'}= \left\lbrace n\le x|\quad \exists k\in \left\lbrace1,2,3 \right\rbrace \mbox{ and } \exists p>r \mbox{ such that } p^m\|F_k(n), |1-g_k(p^m)|>\frac{1}{2}\right\rbrace.
\end{eqnarray*}
Decompose $S_2$ into two sums
\begin{eqnarray*}
S_2&=& \frac{1}{x}\sum_{n\in N_r^{'}}\left|g_{1r}^{*}(F_1(n))g_{2r}^{*}(F_2(n))g_{3r}^{*}(F_3(n))-P^{'}(r,x)\right|\\
&+&\frac{1}{x}\sum_{n\notin N_r^{'}}\left|g_{1r}^{*}(F_1(n))g_{2r}^{*}(F_2(n))g_{3r}^{*}(F_3(n))-P^{'}(r,x)\right|\\
&:=& S_5+S_6.
\end{eqnarray*}
Let us put,
\[\eta_j(p):=\sum\limits_{m=1}^{\infty}\frac{\left(g_j(p^m))-g_j(p^{m-1})\right)\varrho_j(p^m)}{p^m}, j=1,2,3.\]
From Lemma $1$, we have
\[|\eta_j(p)|\le 2c_{F_j}\frac{1}{p-1}\le \frac{1}{6} \text{ if } p\ge 1+12c_{F_j}:=p_j.\]
Let, $p_4:=\max (p_1,p_2,p_3).$
So, if $r\ge p_4,$ then
\begin{eqnarray*}
P^{'}(r,x)&=&\prod\limits_{r<p\le x}\left(1+\eta_1(p)+\eta_2(p)+\eta_3(p)\right)\\
&=&\exp\biggl(\sum\limits_{r<p\le x}\left(\eta_1(p)+\eta_2(p)+\eta_3(p)+O\left(|\eta_1(p)|^2+|\eta_2(p)|^2+|\eta_3(p)|^2\right)\right)\biggr)\\
&=&\exp\biggl(\sum\limits_{r<p\le x}\sum\limits_{j=1}^{3}\frac{(g_j(p)-1)\varrho_j(p)}{p}+O\biggl(\sum\limits_{r<p\le x}\frac{1}{p^2}\biggr)\biggr)\ll 1.
\end{eqnarray*}
Without loss of generality we may assume that, $r\ge p_4.$\\
Now
\begin{eqnarray*}
S_5&\ll &\frac{1}{x}\sum\limits_{n\in N_r^{'}}1\ll \frac{1}{x}\sum\limits_{\substack{p^m\le F_j(x)\\|1-g_j(p^m)|>1/2\\p>r}}\left(\frac{x\varrho_j(p^m)}{p^m}+\varrho_j(p^m)\right)\\
&\ll &\frac{1}{x}\sum\limits_{\substack{p^m\le F_j(x)\\|1-g_j(x)|>1/2\\p>r}}\frac{x\varrho_j(p^m)}{p^m}+\frac{1}{x}\sum\limits_{\substack{p^m\le F_j(x)\\|1-g_j(p^m)|>1/2\\p>r}}\varrho_j(p^m)\\
&\ll &\sum\limits_{r<p\le F_j(x)}\frac{|1-g_j(p)|\varrho_j(p)}{p}+\sum\limits_{p>r}\frac{1}{p^2}+\frac{1}{x}\sum\limits_
{\substack{p^m\le F_j(x)\\p>r,m<v_j}}|1-g_j(p^m)|\varrho_j(p^m)+\frac{1}{\log x}\\
&\ll & S^{'}(r,x)+T(x)+(r\log r)^{-1}+\frac{1}{x}C(r)+\frac{1}{\log x}.
\end{eqnarray*}
Since we know that if $\Re(u)\le 0, \Re(v)\le 0$, then
\begin{align}
&\left|\exp(u)-\exp(v)\right|\le |u-v|\quad \text{and} \quad \text{for} \quad |z|\le 1, |\arg(z)|\le \frac{\pi}{2}\\
&\log (1+z)=z+O(|z|^2). 
\end{align}
we have,
\begin{eqnarray*}
S_6&\ll &\frac{1}{x}\sum\limits_{n\le x}\sum\limits_{j=1}^{3}\Biggl|\sum\limits_{\substack{p^m\parallel F_j(n)\\p>r}}(g_j(p^m)-1)-\sum\limits_{\substack{p^m\le x\\p>r}}\frac{g_j(p^m)-1}{p^m}\varrho_j(p^m)\Biggr|\\
&+&\frac{1}{x}\sum\limits_{n\le x}\Biggl|\sum\limits_{\substack{p^m\le x\\p>r}}\sum\limits_{j=1}^{3}\frac{(g_j(p^m)-1)\varrho_j(p^m)}{p^m}-\log P^{'}(r,x)\Biggr|\\
&+&O\Biggl(\frac{1}{x}\sum\limits_{n\le x}\sum\limits_{j=1}^{3}\sum\limits_{\substack{p^m\parallel F_j(x)\\p>r}}\left|g_j(p^m)-1\right|^2\Biggr):=S_{61}+S_{62}+S_{63}.
\end{eqnarray*}
From Lemma $2,$ we have
\begin{eqnarray*}
S_{61}&\ll &\sum\limits_{j=1}^{3}\Biggl(\sum\limits_{\substack{p^m\le F_j(x)\\p>r}}\frac{\left|g_j(p^m)-1\right|^2\varrho_j(p^m)}{p^m}\Biggr)^{1/2}+\frac{1}{x}C(r)+\frac{1}{\log x}\\
&\ll &(S^{'}(r,x))^{1/2}+(T(x))^{1/2}+(r\log r)^{-1/2}+\frac{1}{x}C(r)+\frac{1}{\log x}. 
\end{eqnarray*}
\begin{eqnarray*}
S_{62}&=&\biggl|\sum\limits_{j=1}^{3}\sum\limits_{r<p\le x}\frac{(g_j(p)-1)\varrho_j(p)}{p}+O\biggl(\sum\limits_{r<p}\frac{1}{p^2}\biggr)-\sum\limits_{r<p\le x}\log w_{p}^{'}\biggr|\\
&=&\biggl|\sum\limits_{j=1}^{3}\sum\limits_{r<p\le x}\frac{(g_j(p)-1)\varrho_j(p)}{p}+O\biggl(\sum\limits_{r<p}\frac{1}{p^2}\biggr)-\sum\limits_{r<p\le x}\sum\limits_{j=1}^{3}\frac{\left(g_j(p)-1\right)\varrho_j(p)}{p}\biggr|\\
&\ll &\sum\limits_{p>r}\frac{1}{p^2}\ll (r\log r)^{-1}.
\end{eqnarray*}
\begin{eqnarray*}
S_{63}&\ll &\sum\limits_{j=1}^{3}\sum\limits_{\substack{p^m\le F_j(x)\\p>r}}\frac{\left|g_j(p^m)-1\right|^2\varrho_j(p^m)}{p^m}+\frac{1}{x}C(r)+\frac{1}{\log x}\\
&\ll &S^{'}(r,x)+T(x)+(r\log r)^{-1}+\frac{1}{x}C(r)+\frac{1}{\log x}.
\end{eqnarray*}
Combining all these estimates for all $1-\frac{1}{v_1+v_2+v_3}<\alpha <1,$ we have
\begin{eqnarray*}
M_x^{''}(g_1,g_2,g_3)-P^{'}(x)&\ll & \left(F_1(x)F_2(x)F_3(x)\right)^{1-\alpha}\exp \left(c\frac{r^\alpha}{\log r}\right)+(S^{'}(r,x))^{1/2}\\
&+&(T(x))^{1/2}+(r\log r)^{-1/2}+\frac{1}{x}C(r)+\frac{1}{\log x}.
\end{eqnarray*}
which proves the required Theorem.
\section{\textbf{Proof of Theorem $\bf{3}$ and $\bf{4}$}}

We begin with some lemmas.The first lemma will be used to prove Theorem $3.$
\begin{lem}[\cite{DAV},Theorem $1$]
For any given $K>0$,
\[\sum\limits_{n\le x}\mu(n)\exp(2\pi in\theta)=O\left(\frac{x}{(\log x)^K}\right),\]
uniformly in $\theta.$
\end{lem}
The next lemmas will be used to prove Theorem $4.$
\begin{lem}[\cite{ELL2},Theorem $2$]
Let $g$ be a multiplicative function whose modulus does not exceed $1.$ Then there is a real $\tau,$ $|\tau|\le (\log x)^{1/19},$ such that
\begin{align}
\sum\limits_{\substack{n\le x\\(n,D)=1}}g(n)=\theta_{\tau}(D)\sum\limits_{n\le x}g(n)+O\left(\frac{x(\log\log 3d)^2}{(\log x)^{1/19}}\right)
\end{align}
holds uniformly for $x\ge 2$ and odd integers $D.$ If, in addition, the condition $(11)$ is satisfied then $(31)$ holds for even integers as well. For real-valued $g$ we may set $\tau =0.$
\end{lem}
The following lemma is a special case of a t(heorem of Wolke \cite{WOL}.
\begin{lem}[\cite{WOL},Theorem $1$]
Let $g$ is as above and $g$ is a good function. Then for given any $A>0$ there is a corresponding $A_1>0,$ possibly depending on $g,$ such that for $x\ge 2$ and $Q=x^{1/2}(\log x)^{-A_1},$ we have
\[\sum\limits_{d\le Q}\max_{(l,d)=1}\max_{u\le x}\Biggl|\sum\limits_{\substack{n\le u\\n\equiv l\pmod d}}g(n)-\frac{1}{\phi(d)}\sum\limits_{\substack{n\le u\\(n,d)=1}}g(n)\Biggr|\ll \frac{x}{(\log x)^A}\]
In case $-\tau \in \mathbb{N}$ or $\tau =0$ then
\[\sum\limits_{d\le Q}\max_{l}\max_{u\le x}\Biggl|\sum\limits_{\substack{n\le u\\n\equiv l\pmod d}}g(n)\Biggr|\ll \frac{x}{(\log x)^A}.\]
\end{lem}
The following lemma is a two dimensional version of standard Cauchy-Schwarz inequality:
\begin{lem}
If $x_j,x_k$ and $c_{jk}$ are non-negetive real numbers, then
\[\sum\limits_{j\le y}\sum\limits_{k\le y}x_jx_kc_{jk}\le \biggl(\sum\limits_{j\le y}\sum\limits_{k\le y}x_j^2x_k^2c_{jk}\biggr)^{1/2}\biggl(\sum\limits_{j\le y}\sum\limits_{k\le y}c_{jk}\biggr)^{1/2}.\]
\end{lem}
\begin{proof}
By applying standard Cauchy-Schwarz over $k,$ we have
\[\sum\limits_{j\le y}\sum\limits_{k\le y}x_jx_kc_{jk}\le \sum\limits_{j\le y}\biggl(\sum\limits_{k\le y}x_j^2x_k^2c_{jk}\biggr)^{1/2}\biggl(\sum\limits_{k\le y}c_{jk}\biggr)^{1/2}:=\sum\limits_{j\le y}a_jb_j.\]
Again by applying cauchy-Schwarz over $j,$ we have
\[\sum\limits_{j\le y}a_jb_j\le \biggl(\sum\limits_{j\le y}a_j^2\biggr)^{1/2}\biggl(\sum\limits_{j\le y}b_j^2\biggr)^{1/2}=\biggl(\sum\limits_{j\le y}\sum\limits_{k\le y}x_j^2x_k^2c_{jk}\biggr)^{1/2}\biggl(\sum\limits_{j\le y}\sum\limits_{k\le y}c_{jk}\biggr)^{1/2}.\]
\end{proof}
\subsection{\textbf{Proof of Theorem $\bf{3}$}}
Let us put,
\[R(r,x)=\prod\limits_{r<p\le x}\biggl\lbrace 1-\frac{2}{p}+\left(1-\frac{1}{p}\right)\sum\limits_{m=1}^{\infty}\frac{g_1(p^m)+g_2(p^m)}{p^m}\biggr\rbrace.\]\\
It is easy to see that, $|R(r,x)|\le 1.$ Therefore,
\begin{eqnarray*}
M_{x}^{'}(g_{1}, g_{2}, \mu)=R(r,x)\frac{1}{x}\sum_{n\le x}g_{1r}(n+3)g_{2r}(n+2)\mu(n+1)\\ 
+\frac{1}{x}\sum_{n\le x}g_{1r}(n+3)g_{2r}(n+2)\mu(n+1)\left(g_{1r}^{\ast}(n+3)g_{2r}^{\ast}(n+2)-R(r,x)\right).
\end{eqnarray*}
So,
\begin{eqnarray*}
x\left|M_x^{'}(g_1,g_2,\mu)\right| &\le & \biggl|\sum\limits_{n\le x}g_{1r}(n+3)g_{2r}(n+2)\mu(n+1)\biggr|+\sum\limits_{n\le x}\biggl|g_{1r}^{*}(n+3)g_{2r}^{*}(n+2)-R(r,x)\biggr|\\
&:=&T_1+T_2.
\end{eqnarray*}
\subsection*{Estimation of $T_1$}
\begin{align*}
&\sum\limits_{n\le x}g_{1r}(n+3)g_{2r}(n+2)\mu(n+1)= \sum\limits_{n\le x}\sum\limits_{d_1|n+3}\sum\limits_{d_2|n+2}h_{1r}(d_1)h_{2r}(d_2)\mu(n+1)\\
&=\sum\limits_{d_1\le x+3}\sum\limits_{d_2\le x+2}h_{1r}(d_1)h_{2r}(d_2)\sum\limits_{\substack{2\le n\le x+1\\ d_1|n+2\\d_2|n+1}}\mu(n)
=\sum\limits_{d_1\le x+3}\sum\limits_{\substack{d_2\le x+2\\(d_1,d_2)=1}}h_{1r}(d_2)h_{2r}(d_2)\sum\limits_{\substack{2\le n\le x+1\\ n\equiv v(d_1 d_2)}}\mu(n)\\
&=\sum\limits_{d_1\le y}h_{1r}(d_1)\sum\limits_{\substack{d_2\le y\\ (d_1,d_2)=1}}h_{2r}(d_2)\sum\limits_{\substack{2\le n\le x+1\\ n\equiv v(d_1 d_2)}}\mu(n)
+\sum\limits_{k=1}^{2}\sum\limits_{\substack{d_j\le x+4-j\\ j=1,2\\
d_k>y\\(d_1,d_2)=1}}h_{1r}(d_1)h_{2r}(d_2)\sum\limits_{\substack{2\le n\le x+1\\n\equiv v(d_1 d_2)}}\mu(n),
\end{align*}
where $v$ is unique solution of the system of linear congruence $n\equiv -2(d_1), n\equiv -1(d_2), 0\le v \le d_1d_2-1$ and $y:=\log x.$ \\
So,
\begin{eqnarray*}
T_1&\ll & \sum\limits_{\substack{d_j\le y\\j=1,2}}|h_{1r}(d_1)h_{2r}(d_2)|\Biggl|\sum\limits_{\substack{n\le x+1\\n\equiv v(d_1d_2)}}\mu(n)\Biggr|+\sum\limits_{k=1}^{2}\sum\limits_{\substack{d_j\le x+4-j\\j=1,2\\d_k>y}}|h_{1r}(d_1)h_{2r}(d_2)|\left(\frac{x}{d_1d_2}+1\right)\\
&:=&T_{11}+T_{12}.
\end{eqnarray*}
From Lemma $4$ we have,
\begin{align*}
&\sum\limits_{\substack{n\le x+1\\n\equiv v(d_1d_2)}}\mu(n)=\sum\limits_{n\le x+1}\mu(n)\frac{1}{d_1d_2}\sum\limits_{l=1}^{d_1d_2}\exp\left(\frac{(n-v)l}{d_1d_2}\right)\\
&=\frac{1}{d_1d_2}\sum\limits_{l=1}^{d_1d_2}\exp\left(\frac{-vl}{d_1d_2}\right)\sum\limits_{n\le x+1}\mu(n)\exp\left(\frac{nl}{d_1d_2}\right)\ll \frac{x}{(\log x)^K}.
\end{align*}
So,
\begin{eqnarray*}
T_{11}&\ll &\frac{x}{(\log x)^K}\sum\limits_{d_1\le y}\sum\limits_{d_2\le y}|h_{1r}(d_1)h_{2r}(d_2)|\ll \frac{x}{(\log x)^K}y^4 \sum\limits_{d_1=1}^{\infty}\sum\limits_{d_2=1}^{\infty}\frac{|h_{1r}(d_1)h_{2r}(d_2)|}{d_1^2 d_2^2}\\
&\ll &\frac{x}{\log x}, \text{if } K\ge 5.
\end{eqnarray*}
Now from the following two estimations
\begin{eqnarray} 
\sum\limits_{d_j=1}^{\infty}\frac{|h_{jr}(d_j)|}{d_{j}^{\alpha}}&=& \prod\limits_{p\le r}\left(1+\sum\limits_{m=1}^{\infty}\frac{|h_{jr}(p^m)|}{p^{m\alpha}}\right)
\le \prod\limits_{p\le r}\left(1+\frac{2}{p^\alpha -1}\right)\\
&\le &\exp \biggl(c_1 \sum\limits_{p\le r}\frac{1}{p^\alpha}\biggr)\le \exp \left(c_2\frac{r^{1-\alpha}}{\log r}\right)\nonumber. 
\end{eqnarray}
and
\begin{align}
 \sum\limits_{d_j=1}^{\infty}\frac{|h_{jr}(d_j)|}{d_j}\le \exp\biggl(c_3 \sum\limits_{p\le r}\frac{1}{p}\biggr)\ll (\log r)^{c_4}
 \end{align} 
 We have,
\begin{eqnarray*}
T_{12}&\ll & x\sum\limits_{\substack{d_j\le x+4-j\\j=1,2\\d_k>y}}\frac{|h_{1r}(d_1)h_{2r}(d_2)|}{d_1d_2}+\sum\limits_{\substack{d_i\le x+4-j\\j=1,2\\d_k>y}}|h_{1r}(d_1)h_{2r}(d_2)|\\
&\ll &xy^{-\gamma}\exp\left(c_{2}\frac{r^{\gamma}}{\log r}\right)(\log r)^{c_4}+x^{2\alpha}\exp\left(2c_{2}\frac{r^{1-\alpha}}{\log r}\right)\\
&\ll &xy^{-\gamma}\exp\left(c_{5}\frac{r^{\gamma}}{\log r}\right)+x^{2\alpha}\exp\left(2c_{2}\frac{r^{1-\alpha}}{\log r}\right).
\end{eqnarray*}
Taking $1-\alpha = \gamma=\frac{2}{3}$ we have,
\[T_{12}\ll xy^{-\frac{2}{3}}\exp\left(c_{5}\frac{r^{2/3}}{\log r}\right)+x^{2/3}\exp\left(2c_{2}\frac{r^{2/3}}{\log r}\right).\]
Putting $r=(\log\log x)^{3/2}$ we have
\[T_{12}\ll \frac{x}{y^{2/3}}(\log x)^{1/6}+xx^{-1/3}(\log x)^{1/6}\ll \frac{x}{(\log x)^{1/2}}.\]
So as $x\rightarrow \infty,$
\[T_1=o(x).\]
\subsection*{Estimation of $T_2$}
To get an estimate of $T_2$ we will use an technique of R.Warlimont \cite{WAR}.\\
Now, let
\begin{eqnarray*}
N_r= \left\lbrace n\le x |\quad \exists  j\in \left\lbrace1,2\right\rbrace \mbox{ and }\exists  p>r \mbox{ such that } p^m\|n+4-j, |1-g_j(p^m)|>\frac{1}{2}\right\rbrace
\end{eqnarray*}
Decompose $T_2$ into two sums
\begin{align*}
T_2&= \frac{1}{x}\sum_{n\in N_{r}}\left|g_{1r}^{*}(n+3)g_{2r}^{*}(n+2)-R(r,x)\right|
+\frac{1}{x}\sum_{n\notin N_{r}}\left|g_{1r}^{*}(n+3)g_{2r}^{*}(n+2)-R(r,x)\right|\\
&:= T_5+T_6.
\end{align*}
Now,
\begin{eqnarray*}
T_5\ll \frac{1}{x}\sum_{j=1}^{2}\sum_{r<p\le x+4-j}\frac{x+4-j}{p}|1-g_j(p)|^2+ \sum_{j=1}^{2}\sum_{p>r}\sum_{m\ge 2}\frac{1}{p^m}\ll S^{''}(r,x)+(r\log r)^{-1}
\end{eqnarray*}
From $(26)$ and $(27),$ we have
\begin{align*}
&T_6 \le \frac{1}{x}\sum_{j=1}^{2}\sum_{n\le x}\biggl|\sum_{\substack{p^m\parallel n+4-j\\ p>r}}\left(g_j(p^m)-1\right)-\sum_{\substack{p^m\le x\\ p>r}}\frac{g_j(p^m)-1}{p^m}\biggr|+ \frac{1}{x}\biggl|\sum_{\substack{p^m\le x\\ p>r}}\sum\limits_{j=1}^{2}\frac{g_j(p^m)-1}{p^m}-\log R(r,x)\biggr|\\
&+ O\biggl(\frac{1}{x}\sum_{n\le x}\sum_{j=1}^{2}\sum_{\substack{p^m\parallel n+4-j\\ p>r}}\left|g_j(p^m)-1\right|^2\biggr):=T_7+T_8+T_9.
\end{align*}
Now by Cauchy-Schwarz inequality and Tur\'{a}n-Kubilius inequality $($\cite{ELL}, Lemma $4.4)$, we have
\begin{eqnarray*}
T_7&\ll &\bigg(\sum_{j=1}^{2}\sum_{\substack{p^m\le x+4-j\\p>r}}\left|g_j(p^m)-1\right|^2\biggr)^\frac{1}{2}+\frac{1}{x}\ll \biggl(\sum_{j=1}^{2}\sum_{r<p\le x+4-j}\left|g_j(p)-1\right|^2\biggr)^\frac{1}{2}+\biggl(\sum_{p>r}\frac{1}{p^2}\biggr)^\frac{1}{2}+\frac{1}{x}\\
&\ll & \left(S^{''}(r,x)\right)^\frac{1}{2}+(r\log r)^{-\frac{1}{2}}+x^{-1}.
 \end{eqnarray*}
 Now similar to estimation of $S_{62},$ we have
 \begin{eqnarray*}
 T_8\ll \biggl|\sum_{r<p\le x}\sum\limits_{j=1}^{2}\frac{g_j(p^m)-1}{p^m}-\log R(r,x)\biggr|
 \ll \sum_{p>r}\frac{1}{p^2}\ll (r\log r)^{-1}.
\end{eqnarray*}
and
\begin{eqnarray*}
T_9 \ll \frac{1}{x}\biggl\lbrace \sum_{j=1}^{2}\sum_{\substack{p^m\le x+4-j\\ p>r}}\frac{\left|g_j(p^m)-1\right|^2}{p^m}\biggr\rbrace
\ll  S^{''}(r,x)+(r\log r)^{-1}.
\end{eqnarray*}
Combining above calculations, we have
\[T_2\ll (r\log r)^{-1/2}+(S(r,x))^{1/2}+x^{-1}.\]
By the above choice of $r$ and from $(10)$ we have as $x\rightarrow \infty,$
\[T_2=o(x).\]
Which proves the required Theorem.
\subsection{\textbf{Proof of Theorem $\bf{4}$}}
\begin{eqnarray*}
R&:=&M_x^{'}(g_1,g_2,g_3)-M_x(g_3)P^{''}(r,x)Q(r)=P^{''}(r,x)\left(M_x^{'}(g_{1r},g_{2r},g_3)-M_x(g_3)Q(r)\right)\\
&+&\frac{1}{x}\sum\limits_{n\le x}g_{1r}(n+3)g_{2r}(n+2)g_3(n+1)\left(g_{1r}^{\ast}(n+3)g_{2r}^{\ast}(n+2)-P^{''}(r,x)\right)
\end{eqnarray*}
It is easy to see that $|P^{''}(r,x)|\le 1.$ Therefore
\begin{eqnarray*}
R&\ll &\left|M_x^{'}(g_{1r},g_{2r},g_3)-M_x(g_3)Q(r)\right|+\frac{1}{x}\sum\limits_{n\le x}\left|g_{1r}^{\ast}(n+3)g_{2r}^{\ast}(n+2)-P^{''}(r,x)\right|\\
&:=& U_1+U_2.
\end{eqnarray*}
\subsection*{Estimation of $U_1$}
\begin{align*}
&M_x^{'}(g_{1r},g_{2r},g_3)=\frac{1}{x}\sum\limits_{n\le x}\sum\limits_{d_1|n+3}\sum\limits_{d_2|n+2}h_{1r}(d_1)h_{2r}(d_2)g_3(n+1)\\
&=\frac{1}{x}\sum\limits_{d_1\le x+3}h_{1r}(d_1)\sum\limits_{\substack{d_2\le x+2\\(d_1,d_2)=1}}h_{2r}(d_2)\sum\limits_{\substack{2\le n\le x+1\\n\equiv v(d_1d_2)}}g_3(n)\\
&=\frac{1}{x}\sum\limits_{\substack{d_j\le y\\j=1,2\\(d_1,d_2)=1}}h_{1r}(d_1)h_{2r}(d_2)\sum\limits_{\substack{2\le n\le x+1\\n\equiv v(d_1d_2)}}g_3(n)+\frac{1}{x}\sum\limits_{k=1}^{2}\sum\limits_{\substack{d_j\le x+4-j\\j=1,2\\d_k>y}}h_{1r}(d_1)h_{2r}(d_2)\left(\frac{x}{d_1d_2}+1\right)\\
&:=P_2+U_{11},
\end{align*}
where $v$ is the unique solution of the system $n\equiv-2(d_1), n\equiv-1(d_2)$ and $y:=x^{1/4}(\log x)^{-\frac{\beta}{2}}.$\\
From $(29),(30),$  for $0<\alpha,\gamma<1,$ we have
\begin{eqnarray*}
U_{11}&\ll &y^{-\gamma}\exp\left(c_{6}\frac{r^\gamma}{\log r}\right)+x^{1-2\alpha}\exp\left(2c_2\frac{r^\alpha}{\log r}\right)\\
&\ll &x^{-\gamma/4}(\log x)^{\frac{\gamma \beta}{2}}\exp\left(c_{6}\frac{r^\gamma}{\log r}\right)+x^{1-2\alpha}\exp\left(2c_2\frac{r^\alpha}{\log r}\right).
\end{eqnarray*}
\begin{align*}
&P_2=\frac{1}{x}\sum\limits_{d_1\le y}\sum\limits_{\substack{d_2\le y\\(d_1,d_2)=1}}\frac{h_{1r}(d_1)h_{2r}(d_2)}{\phi(d_1d_2)}\sum\limits_{\substack{n\le x\\(n,d_1d_2)=1}}g_3(n)+O\Biggl(\frac{1}{x}\sum\limits_{\substack{d_j\le y\\j=1,2}}|h_{1r}(d_2)h_{2r}(d_2)|\Biggr)\\
&+O\Biggl(\frac{1}{x}\sum\limits_{\substack{d_j\le y\\j=1,2}}|h_{1r}(d_1)h_{2r}(d_2)|\Biggl|\sum\limits_{\substack{n\le x\\n\equiv v(d_1d_2)}}g_3(n)-\frac{1}{\phi(d_1d_2)}\sum\limits_{\substack{n\le x\\(n,d_1d_2)=1}}g_3(n)\Biggr|\Biggr):=P_3+U_{12}+U_{13}.
\end{align*}
By Lemma $6$ and Lemma $7,$ we have
\begin{align*}
U_{13}&\ll \frac{1}{x}\Biggl(\sum\limits_{l\le y^2}\Biggl|\sum\limits_{\substack{n\le x\\n\equiv v(l)}}g_3(n)-\frac{1}{\phi(l)}\sum\limits_{\substack{n\le x\\(n,l)=1}}g_3(n)\Biggr|\Biggr)^{1/2}\\
&\Biggl(\sum\limits_{\substack{d_j\le y\\j=1,2}}|h_{1r}(d_1)|^2|h_{2r}(d_2)|^2\Biggl|\sum\limits_{\substack{n\le x\\n\equiv v(d_1d_2)}}g_3(n)-\frac{1}{\phi(d_1d_2)}\sum\limits_{\substack{n\le x\\(n,d_1d_2)=1}}g_3(n)\Biggr|\Biggr)^{1/2}\\
&\ll \frac{1}{(\log x)^{A/2}}\Biggl(\sum\limits_{\substack{d_j\le y\\j=1,2}}\frac{|h_{1r}(d_1)|^2|h_{2r}(d_2)|^2}{\phi(d_1)\phi(d_2)}\Biggr)^{1/2}.
\end{align*}
Since
\[\sum\limits_{d\le y}\frac{|h(d)|^2}{\phi(d)}\le \exp\biggl(c_7\sum\limits_{p\le r}\frac{1}{p}\biggr)\le (\log r)^{c_8},\]
we have
\[U_{13}\ll \frac{(\log r)^{c_8}}{(\log x)^{A/2}}.\]
Now from $(29)$ we get
\[U_{12}\ll \frac{y^{2(1-\alpha)}}{x}\sum\limits_{\substack{d_j=1\\j=1,2}}^{\infty}\frac{|h_{1r}(d_1)h_{2r}(d_2)|}{d_1^{1-\alpha}d_2^{1-\alpha}}\ll x^{-\frac{1}{2}(1+\alpha)}(\log x)^{-\beta (1-\alpha)}\exp\left(2c_2\frac{r^\alpha}{\log r}\right).\]
Using Lemma $5$, we get
\begin{eqnarray*}
P_3&=&\frac{1}{x}\sum\limits_{d_1\le y}\sum\limits_{\substack{d_2\le y\\(d_1,d_2)=1}}\frac{h_{1r}(d_1)h_{2r}(d_2)}{\phi(d_1d_2)}\theta_{\tau}(d_1d_2)\sum\limits_{n\le x}g_3(n)\\
&+&O\Biggl(\sum\limits_{\substack{d_j\le y\\j=1,2\\(d_1,d_2)=1}}\frac{|h_{1r}(d_1)h_{2r}(d_2)|}{\phi(d_1)\phi(d_2)}\frac{(\log\log 3d_1d_2)^2}{(\log x)^{1/19}}\Biggr):=P_4+U_{14}.
\end{eqnarray*}
Now,
\begin{align*}
&\sum\limits_{d_1\le y}\sum\limits_{\substack{d_2\le y\\(d_1,d_2)=1}}\frac{h_{1r}(d_1)h_{2r}(d_2)}{\phi(d_1)\phi(d_2)}\theta_{\tau}(d_1)\theta_{\tau}(d_2)=\sum\limits_{d_1=1}^
{\infty}\sum\limits_{\substack{d_2=1\\(d_1,d_2)=1}}^{\infty}
\frac{h_{1r}(d_1)h_{2r}(d_2)}{\phi(d_1)\phi(d_2)}\theta_{\tau}(d_1)\theta_{\tau}(d_2)\\
&+O\Biggl(\sum\limits_{k=1}^{2}
\sum\limits_{\substack{d_j\le y\\j=1,2\\d_k>y}}\frac{h_{1r}(d_1)h_{2r}(d_2)}{\phi(d_1)\phi(d_2)}\theta_{\tau}(d_1)\theta_{\tau}(d_2)\Biggr):=P_5+U_{15}.
\end{align*}
Now from the following two estimation
\[ \sum\limits_{d>y}\frac{|h_{jr}(d)\theta_{\tau}(d)|}{\phi(d)}\le y^{-\alpha}\exp\biggl(c_{9} \sum\limits_{p\le r}\frac{1}{p^{1-\alpha}}\biggr)\ll x^{-\frac{1-\alpha}{4}}(\log x)^{\frac{\beta (1-\alpha)}{2}}\exp\left(c_{10}\frac{r^{\alpha}}{\log r}\right).\]
and 
\[\sum\limits_{d=1}^{\infty}\frac{|h_{jr}(d)\theta_{\tau}(d)|}{\phi(d)}\le \exp\biggl(c_{11}\sum\limits_{p\le r}\frac{1}{p}\biggr)\le (\log r)^{c_{12}}.\]
we can sat that
\[U_{15}\ll x^{-\frac{1-\alpha}{4}}(\log x)^{\frac{\beta (1-\alpha)}{2}}\exp\left(c_{13}\frac{r^{\alpha}}{\log r}\right).\]
\begin{eqnarray*}
P_5&=&\prod_{p\le r}\biggl(1+\sum\limits_{m=1}^{\infty}\frac{\left(h_1(p^m)+h_2(p^m)\right)
\theta_{\tau}(p^m)}{\phi(p^m)}\biggr)=\prod_{p\le r}\biggl(1-\frac{2\theta_{\tau}(p)}{p-1}+\theta_{\tau}(p)\sum\limits_{m=1}^{\infty}\frac{g_1(p^m)+g_2(p^m)}{p^m}\biggr)\\
&=&Q(r).
\end{eqnarray*}
Since
\begin{eqnarray*}
\sum\limits_{d=1}^{\infty}\frac{|h_{jr}(d)|\log\log d}{\phi(d)}&=&\prod\limits_{p\le r}\biggl(1+\sum\limits_{\alpha=1}^{\infty}\frac{|h_{jr}(p^\alpha)|\alpha\log\log p}{\phi(p^m)}\biggr)\ll \prod\limits_{p\le r}\left(1+\frac{2p\log\log p}{(p-1)^2}\right)\\
&\ll &\exp\biggl(c_{14}\sum\limits_{p\le r}\frac{\log\log p}{p}\biggr)\ll \exp\left(c_{15}(\log\log r)^2\right),
\end{eqnarray*}
then we can say that
\[U_{14}\ll \frac{\exp\left(2c_{15}(\log\log r)^2\right)}{(\log x)^{1/19}}.\]
From the similar calculation as Estimation of $T_2$, we have
\[U_2\ll (S^{''}(r,x))^{1/2}+(r\log r)^{-1/2}.\]
Combining these results, we get
\begin{align*}
&R\ll \left(x^{-\frac{\gamma}{4}}(\log x)^{\gamma \beta}+x^{\frac{\alpha-1}{4}}(\log x)^{\frac{\beta (1-\alpha)}{2}}\right)\exp\left(c_{16}\frac{r^\gamma}{\log r}\right)+x^{1-2\alpha}\exp\left(c_{17}\frac{r^\alpha}{\log r}\right)+\frac{(\log r)^{c_8}}{(\log x)^{\frac{A}{2}}}\\
&+x^{-\frac{(1+\alpha)}{2}}(\log x)^{\beta (1-\alpha)}\exp\left(c_{18}\frac{r^{\alpha}}{\log r}\right)+\frac{\exp\left(2c_{15}(\log\log r)^2\right)}{(\log x)^{1/19}}+(S^{''}(r,x))^{\frac{1}{2}}+(r\log r)^{-\frac{1}{2}}.
\end{align*}
By choosing $\alpha=\gamma,$ we get the required theorem. 
\section{\textbf{Proof of Therorem $\bf{5}$}}

Let us put,
\[T(r,x)=\prod\limits_{r<p\le x}\biggl\lbrace 1-\frac{1}{p}+\left(1-\frac{1}{p}\right)\sum\limits_{m=1}^{\infty}\frac{g_1(p^m)}{p^m}\biggr\rbrace.\]
Now,
\begin{eqnarray*}
\sum\limits_{n\le x}g_1(n+3)\mu(n+2)\mu(n+1)&=&T(r,x)\sum\limits_{n\le x}g_{1r}(n+3)\mu(n+2)\mu(n+1)\\
&+& \sum\limits_{n\le x}g_{1r}(n+3)\mu(n+2)\mu(n+1)\left(g_{1r}^{*}(n+3)-T(r,x)\right). 
\end{eqnarray*}
It is easy to see that $|T(r,x)|\le 1.$Therefore,
\begin{eqnarray*}
\Bigl|\sum\limits_{n\le x}g_1(n+3)\mu(n+2)\mu(n+1)\Bigr|&\le &\Bigl|\sum\limits_{n\le x}g_{1r}(n+3)\mu(n+2)\mu(n+1)\Bigr|+\sum\limits_{n\le x}\left| g_{1r}^{*}(n+3)-T(r,x)\right|\\
&:=&V_1 + V_2.
\end{eqnarray*}
\subsection*{Estimation of $V_1$}
\begin{eqnarray*}
V_1= \biggl|\sum\limits_{n\le x}\sum\limits_{d|n+3}h_{1r}(d)\mu(n+2)\mu(n+1)\biggr|=\biggl|\sum\limits_{d\le x+3}h_{1r}(d)\sum\limits_{\substack{n\le x\\ d|n+3}}\mu(n+2)\mu(n+1)\biggr|\\
\le \biggl|\sum\limits_{d\le y}h_{1r}(d)\sum\limits_{\substack{n\le x\\ n\equiv -3(d)}}\mu(n+2)\mu(n+1)\biggr|+\biggl|\sum\limits_{y<d\le x+3}h_{1r}(d)\sum\limits_{\substack{n\le x\\ n\equiv -3(d)}}\mu(n+2)\mu(n+1)\biggr|:=V_{11}+V_{12},
\end{eqnarray*}
where $y:=\log x.$\\
Under Assumption $1$, we have
\begin{align*}
&\sum\limits_{\substack{n\le x\\ n\equiv -3(d)}}\mu(n+2)\mu(n+1)=\sum\limits_{n\le x}\mu(n+2)\mu(n+1)\frac{1}{d}\sum\limits_{l=1}^{d}\exp\left(\frac{(n-3)l}{d}\right)\\
&=\frac{1}{d}\sum\limits_{l=1}^{d}\exp\left(\frac{-3l}{d}\right)\sum\limits_{n\le x}\mu(n+2)\mu(n+1)\exp\left(\frac{nl}{d}\right)\ll  \frac{x}{(\log x)^A}. 
\end{align*}
By choosing $A=3$ we have,
\begin{eqnarray*}
V_{11}&\ll & \frac{x}{(\log x)^3}\sum\limits_{d\le y}|h_{1r}(d)|\ll \frac{xy^2}{(\log x)^3}\sum\limits_{d=1}^{\infty}\frac{|h_{1r}(d)|}{d^2}\ll  \frac{x}{\log x}.
\end{eqnarray*}
Now using $(29),$ for $0<\alpha <1,$ we have
\begin{eqnarray*}
V_{12}&\ll &\sum\limits_{y<d\le x+3}|h_{1r}(d)|\left(\frac{x}{d}+1\right)\ll \sum\limits_{y<d\le x+3}|h_{1r}(d)|\ll (x+3)\sum\limits_{d>y}\frac{|h_{1r}(d)|}{d}\\
&\ll & \frac{(x+3)}{y^\alpha} \sum\limits_{d=1}^{\infty}\frac{|h_{1r}(d)|}{d^{1-\alpha}}
\ll \frac{x}{y^\alpha}\exp\left(c_{2}\frac{r^\alpha}{\log r}\right)
\end{eqnarray*}
By taking $r=(\log \log x)^\frac{1}{\alpha}$, we have
\[V_{12}\ll xy^{-\alpha}y^{\frac{\alpha}{2}}=\frac{x}{(\log x)^{\frac{\alpha}{2}}}.\]
So as $x\rightarrow \infty$ we have,
\[V_1=o(x).\]
From the similar calculation as Estimation of $T_2$, we have
\[V_2\ll (r\log r)^{-\frac{1}{2}}+\biggl(\sum\limits_{r<p\le x+3}\frac{|g_1(p)-1|^2}{p}\biggr)^{1/2}.\]
From $(12)$ and $r=(\log\log x)^{\frac{1}{\alpha}}$ we have as $x\rightarrow \infty,$ 
\[V_2=o(x).\]
Which proves the required Theorem.
\section{\textbf{Proof of Applications}}
\subsection*{Proof of Application $1$}
We need the following lemma.
\begin{lem}[\cite{TEN}]
Let $\left\lbrace u_{n}\right\rbrace_{n=1}^{\infty}$ and  $\left\lbrace v_n\right\rbrace_{n=1}^{\infty}$ be two complex sequences such that 
\[\sum\limits_{n=1}^{\infty}\left(|u_n|^2+|v_n|\right)<\infty\]
Then,
\[\prod\limits_{n=1}^{\infty}(1+u_n+v_n)<\infty \text{ if and only if } \sum\limits_{n=1}^{\infty}u_n<\infty.\]
\end{lem}
Since $|g_j(p)-1|^2\le 2(1-\Re(g_j(p))),$ from $(5)$ we have,
\begin{align}
\sum\limits_{p}\frac{|g_j(p)-1|^2}{p}< \infty.
\end{align}
So by putting $r=\log x,$ the error term in Theorem $1$  is $o(1)$  as $x\rightarrow \infty.$\\
Now the infinite product $\prod\limits_{p}w_p$ can be written as
\[\prod\limits_{p}w_p=w_2\prod\limits_{p>2}w_p=w_2\prod\limits_{p}\biggl\lbrace1+
\frac{g_1(p)+g_2(p)+g_3(p)-3}{p}+O\left(\frac{1}{p^2}\right)\biggr\rbrace.\]
From $(5)$ and Lemma $8$ we can say that, the above infinite product is convergent.
Hence, by Theorem $1$ Application is proved.
\subsection*{\textbf{Proof of Application $2$}}
We need the Lemma $8$ and the following lemma:
\begin{lem}[\cite{KAT}, Lemma 6]
Let $F(n)$ be a polynomial as above of degree $v\ge 2$. we have the relation:
\[\mathbf{card}\left(n\le x:F(n)\equiv 0\pmod {p^{v-1}}, y_1<p<\infty\right)=o(x),\]
when $y_1$ tends to inifinity as $x\rightarrow \infty.$
\end{lem}
Since $|g_j(p)-1|^2\le 2(1-\Re(g_j(p))),$ from $(13)$ we have,
\begin{align}
\sum\limits_{p}\frac{|g_j(p)-1|^2\varrho_j(p)}{p}< \infty, j=1,2,3.
\end{align}
From $(13)$ and Lemma $8,$ we have $P_1(\gamma)$ and $P_2(\gamma)$ are convergent.\\
From Lemma $9,$ it is easy to see that as $p\rightarrow \infty$ 
\begin{align}
\left(1-g_j(p^{v_j-1})\right)\varrho_j(p^{v_j-1})\rightarrow 0,j=1,2,3.
\end{align}
So by putting $r=\log x$ and from $(14),(32),(33),$ the error term in Theorem $2$ is $o(1)$ as $x\rightarrow \infty.$ 
Hence, by Theorem $2$ Application is proved.
\subsection*{\textbf{Proof of Application $3$}}
 It is easy to see that $\sigma_a$ and $\phi$ are close to $1$ for $a>0.$ 
The remainder term for both sums are estimate from the remainder term of Theorem $1$ by choosing 
\[\beta = 1-\alpha = \min\left(\frac{1}{2k}, \frac{2+c_{19}}{3}\right) \text{ and } r=c_{20}(\log x\log\log x)^{\frac{1}{\beta}},\]
for sufficiently small $c_{19},c_{20}>0.$\\
Hence, by Theorem $1$ Application is proved.        
\subsection*{\textbf{Proof of Application $4$ and $5$}}
Since $\phi$ and $\sigma_a$ are close to $1$ for $a>0$ then by applying Theorem $3$  and Theorem $5$ we get the Application $4$ and Application $5$ respectively. 
\subsection*{\textbf{Proof of Application $6$}}
we need the following lemma:
\begin{lem}
Let $p$ be an odd prime and $(a,p)=1$, then $x^2\equiv a\pmod {p^{k}}$ has exactly two solutions if $a$ is a quadratic residue of $p,$ and no solution if $a$ is quadratic nonresidue of $p.$ Further, if $a$ is odd, then the congruence $x^2\equiv a\pmod 2$is alwalys solvable and has exactly one solution.
\end{lem}
We see that by Lemma $10,$ as $a>0,$  $p\rightarrow \infty$
\begin{align}
\sum\limits_{p}\frac{(g_j(p)-1)\varrho_j(p)}{p}= \sum\limits_{p}\frac{2}{p^{1+a}}< \infty, \quad (g_j(p)-1)\varrho_j(p)= \frac{2}{p^a}\rightarrow 0, 
\end{align}
where $g_j=\sigma_a,j=1,2,3,$ and 
\begin{align}
\sum\limits_{p}\frac{(g_j(p)-1)\varrho_j(p)}{p}= \sum\limits_{p}\frac{-2}{p^2}< \infty, \quad (g_j(p)-1)\varrho_j(p)= \frac{-2}{p}\rightarrow 0, 
\end{align}
where $g_j=\phi, j=1,2,3.$\\
The remainder term for both suma are estimate from the remainder term of Therorem $2$ by choosing
\[\alpha=\frac{5+c_{21}}{6} \quad \text{and} \quad r=c_{22}(\log x\log\log x)^{1/\alpha},\]
for sufficiently small $c_{21},c_{22}>0.$\\
Hence, by Theorem $2$ Application is proved.
\subsection*{\textbf{Proof of Application $7$}}
we need the following lemma:
\begin{lem}[\cite{TEN}]
Let $\left\lbrace F_n\right\rbrace_{n=1}^{\infty}$ be a sequence of distribution functions and $\left\lbrace\phi_n\right\rbrace_{n=1}^{\infty}$ then corresponding sequence of characteristic functions. Then $F_n$ converges weekly to the distribution function $F$ if and only if $\phi_n$ converges pointwise on $\mathbb{R}$ to a function $\phi$ which is continuous at $0.$ In addition, in this case, $\phi$ is the characteristic function of $F$ and the convergence of $\phi_n$ to $\phi$ is uniform on any compact subset.
\end{lem}
The characteristic functions of the distribution $(19)$ equal
\begin{eqnarray}
\frac{1}{[x]}\sum\limits_{n\le x}\exp \left(it(f_1(n+3)+f_2(n+2)+f_3(n+1))\right).
\end{eqnarray}
Since
\begin{align*}
 \sum\limits_{p}\sum\limits_{j=1}^{3}\frac{\exp\left(itf_j(p)\right)-1}{p}=t\sum\limits_{j=1}^{3}\sum\limits_{|f_j(p)|\le 1}\frac{f_j(p)}{p}
+O\biggl(t^2\sum\limits_{j=1}^{3}\biggl( \sum\limits_{|f_j(p)|\le 1}\frac{f_j^2(p)}{p}+\sum\limits_{|f_j(p)|>1}\frac{1}{p}\biggr)\biggr)
\end{align*}
then from the convergence of the series $(16), (17)$ ,$(18)$ and from Lemma $8$ we can say that the infinite product $(20)$ converges for every $t.$ This product is continuous at $t=0$ because it converges uniformly for $|t|\le T$ where $T>0$ is arbitrary.

Since for $j=1,2,3$
\[\sum\limits_{p}\frac{\left|\exp\left(itf_j(p)\right)-1\right|^2}{p}\ll t^2\sum\limits_{|f_j(p)|\le 1}\frac{|f_j(p)|^2}{p}+\sum\limits_{|f_j(p)|>1} 
\frac{1}{p}\]
then from the convergence of $(16)$ and $(17)$ it follows that $S(r,x)\rightarrow 0$ when $r,x\rightarrow \infty.$ Choosing $r=\log x$ in our Theorem $1$ we get that the remainder term disappears when $x\rightarrow \infty.$

Thus the characteristic function $(36)$ has the limit $(20)$ for every real $t$ and this limit is continuous at $t=0.$\\
Therefore, by Lemma $11$ Application is proved.
\subsection*{\textbf{Proof of Application $8$}}
We will use Lemma $8$ and Lemma $11$ to prove this application.
The characteristic functions of the distribution $(24)$ equal
\begin{align}
\frac{1}{[x]}\sum\limits_{n\le x}\exp\left(it\left(f_1(F_1(n))+f_2(F_2(n))+f_3(F_3(n))\right)\right)
\end{align}
Since 
\begin{eqnarray*}
\sum\limits_{p}\sum\limits_{j=1}^{3}\frac{\left(\exp\left(itf_j(p)\right)-1\right)
\varrho_j(p)}{p}=t\sum\limits_{j=1}^{3}\sum\limits_{|f_j(p)|\le 1}\frac{f_j(p)\varrho_j(p)}{p}\\
+O\biggl(t^2\sum\limits_{j=1}^{3}\sum\limits_{|f_j(p)|\le 1}\frac{f_j^2(p)\varrho_j(p)}{p}\biggr)
+ O\biggl(\sum\limits_{j=1}^{3}\sum\limits_{|f_j(p)|>1}\frac{\varrho_j(p)}{p}\biggr).
\end{eqnarray*}
then from the convergence of the series $(21),(22),(23)$ and from Lemma $8$ we can say that $P_1(\gamma)$ and $P_2(\gamma)$ are convergent for every real $t.$ Further, the infinite product $P_1(\gamma)P_2(\gamma)$ is continuous at $t=0$ because it converges uniformly for $|t|\le T$ where $T>0$ is arbitrary.

Since for $j=1,2,3$
\[\sum\limits_{p}\frac{\left|\exp\left(itf_j(p)\right)-1\right|^2\varrho_j(p)}{p}\ll t^2\sum\limits_{|f_j(p)|\le 1}\frac{|f_j(p)|^2\varrho_j(p)}{p}+\sum\limits_{|f_j(p)|>1} 
\frac{\varrho_j(p)}{p}\]
then from the convergence of $(21)$ and $(22)$ it follows that $S^{'}(r,x),T(x)\rightarrow 0$ when $r,x\rightarrow \infty.$\\
Now from $(24)$ and Lemma $9$ it is easy to see that, $\left(\exp\left(itf_j(p^m)-1\right)\varrho_j(p^m)\right)\rightarrow 0$ when $p\rightarrow\infty,m<v_j,j=1,2,3.$ Then $\frac{1}{x}C(r)\rightarrow\ 0$ as $r,x\rightarrow \infty.$ Choosing $r=\log x$ in our Theorem $2$ we get that the remainder term disappears when $x\rightarrow \infty.$
 
Thus the characteristic function $(37)$ has the limit $P_1(\gamma)P_2(\gamma)$ for every real $t$ and this limit is continuous at $t=0.$
 
 Therefore, by Lemma $11$ Application is proved.

\end{document}